\documentclass[10pt,reqno]{amsart}

\usepackage{amssymb}
\usepackage{amsfonts}
\usepackage{amsthm}
\usepackage{amsmath}
\usepackage{bbm}
\usepackage{xcolor}

\numberwithin{equation}{section}
\allowdisplaybreaks

\newtheorem{theorem}{Theorem}[section]
\newtheorem{proposition}[theorem]{Proposition}

\newtheorem{lemma}[theorem]{Lemma}

\newtheorem{conjecture}[theorem]{Conjecture}
\theoremstyle{definition}
\newtheorem{definition}[theorem]{Definition}

\theoremstyle{remark}

\newcommand{\R}{\mathbb{R}}

\newcommand{\N}{\mathbb{N}}

\newcommand{\1}{\mathbbm{1}}

\renewcommand{\hat}{\widehat}
\newcommand{\eps}{\varepsilon}

\newcommand{\scriptD}{\mathcal{D}}
\newcommand{\scriptE}{\mathcal{E}}
\newcommand{\scriptF}{\mathcal{F}}

\newcommand{\scriptH}{\mathcal{H}}

\newcommand{\scriptK}{\mathcal{K}}

\newcommand{\scriptR}{\mathcal{R}}

\newcommand{\scriptV}{\mathcal{V}}

\newcommand{\qtq}[1]{\:\text{#1}\:}

\DeclareMathOperator*{\supp}{supp}

\DeclareMathOperator*{\dist}{dist}

\begin{document}
\title{Fourier restriction above rectangles}
\author{Jeremy Schwend}
\address{Department of Mathematics, University of Wisconsin, Madison, WI 53706}
\email{jschwend@wisc.edu}
\author{Betsy Stovall}
\address{Department of Mathematics, University of Wisconsin, Madison, WI 53706}
\email{stovall@math.wisc.edu}
\begin{abstract}
In this article, we study the problem of obtaining Lebesgue space inequalities for the Fourier restriction operator associated to rectangular pieces of the paraboloid and perturbations thereof.  We state a conjecture for the dependence of the operator norms in these inequalities on the sidelengths of the rectangles, prove that this conjecture follows from (a slight reformulation of the) restriction conjecture for elliptic hypersurfaces, and prove that, if valid, the conjecture is essentially sharp.  Such questions arise naturally in the study of restriction inequalities for degenerate hypersurfaces; we demonstrate this connection by using our positive results to prove new restriction inequalities for a class of hypersurfaces having some additive structure.  
\end{abstract}

\maketitle

\section{Introduction} 

Recent work \cite{BMVadd} establishing bounds for restriction operators associated to higher order surfaces on which the curvature may vanish at some points naturally gives rise to the study of the restriction operator $\scriptR_d^\ell$ associated to the rectangular piece of the paraboloid,
$$
\{(|\xi|^2,\xi):\xi \in Q^\ell\}, \qquad Q^\ell:= \prod_{j=1}^d (-l_j,l_j), \qquad \ell = (l_1,\ldots,l_d) \in (0,\infty]^d,
$$
and perturbations thereof.  

In this article, we consider the problems of establishing finiteness and understanding the dependence on $\ell$ of the $L^p \to L^q$ operator norms of $\scriptR_d^\ell$.  We then apply such results to obtain new, sharp restriction inequalities for a collection of ``degenerate'' hypersurfaces (i.e.\ hypersurfaces whose curvature vanishes on some nonempty set).  We are  motivated by the recent success of the first author \cite{Schwend} (cf.\ \cite{DZ}) in directly deducing sharp estimates for model convolution operators by using a generalization of this approach.    

The natural interpretation of ellipticity in this context leads to a slight generalization of the traditional notion of ellipticity formulated by Tao--Vargas--Vega \cite{TVV}.  We introduce some additional notation, letting 
$$
A^\ell(\xi_1,\ldots,\xi_d):=(l_1\xi_1,\ldots,l_d\xi_d),  \qquad \ell \in (0,\infty)^d,
$$
and 
$$
\text{$\1$}:=(1,\ldots,1).   
$$

\begin{definition} \label{D:elliptic}
Let $\ell \in (0,\infty)^d$ and let $g$ be a $C^{N+2}_{\rm{loc}}$ function on $Q^\ell$ for some $\ell \in (0,\infty]^d$, with $N\geq 0$ sufficiently large, possibly infinite.  Assume that $D^2 g$ is positive definite throughout $Q^\ell$, and let $0 < \eps_0 \leq \tfrac12$.  We say that $g$ is elliptic over $Q^\ell$ (with parameters $N,\eps_0$) if $g(\xi) = |\xi|^2+h(\xi)$, where the perturbation $h$ satisfies $h(0) = 0$, $\nabla h(0) = 0$, $D^2h(0) = 0$, and
$$
\|(D^2 h) \circ A^{\tilde\ell}\|_{C^N(Q^\1)} < \eps_0,
$$
for every bounded $Q^{\tilde\ell}$ contained in $Q^\ell$.
If $g$ is elliptic over $Q^\ell$ with parameters $N,\eps_0$, we will also say that the surface 
$$
\Sigma_g:=\{(g(\xi),\xi):\xi \in Q^\ell\}
$$
is elliptic over $Q^\ell$ (with parameters $N,\eps_0$).  
\end{definition}

This definition of ellipticity is invariant under parabolic rescalings in the sense that $g$ is elliptic over $Q^\ell$ if and only if $\lambda^2g(\lambda^{-1}\cdot)$ is elliptic over $Q^{\lambda \ell}$.  In the special cases that $d=1$ or $\ell = \1 = (1,\ldots,1)$, our definition of ellipticity coincides with that in \cite{TVV}, but ours is strictly more general in the sense that a surface elliptic over some $Q^\ell$ may not be coverable by a bounded (independent of $\ell$) number of surfaces elliptic in the sense of \cite{TVV}.  We will see the utility of this generalization once we turn to applications.  

Associated to a function $g$ elliptic over some $Q^\ell$ are the familiar restriction and extension operators,
\begin{align*}
\scriptR^\ell_g f(\xi) &:= \hat f(g(\xi),\xi), \qquad \xi \in Q^\ell\\
\scriptE^\ell_g f(t,x) &:= \int_{Q^\ell} e^{i(t,x)(g(\xi),\xi)} f(\xi)\, d\xi, \qquad (t,x) \in \R^{1+d}.
\end{align*}
Since these operators are dual to one another, it suffices to state our results for the extension operator.  

Though ellipticity over some $Q^\ell$ is a more general concept than ellipticity in the sense of Tao--Vargas--Vega, the evidence toward the following conjecture seems as strong as that for the corresponding conjecture for elliptic hypersurfaces.  

\begin{conjecture} \label{C:elliptic conj}
For $N$ sufficiently large, $\eps_0>0$ sufficiently large, and $1 \leq p,q \leq \infty$ in the range $q=\frac{d+2}d p' > p$, there exists a constant $C_{p,q,d} < \infty$ such that for any $\ell \in (0,\infty]^d$, $\|\scriptE^\ell_g\|_{L^p \to L^q} \leq C_{p,q,d}$, for any $g$ elliptic over $Q^\ell$ with parameters $N,\eps_0$.  
\end{conjecture}

This conjecture is already verified in the case $d=1$ by Fefferman--Stein \cite{Fefferman} and Zygmund \cite{Zygmund}.  To the authors' read, it seems likely that all recent results, including \cite{GuthI, GuthII, HickmanRogers, Wang}, establishing bounds for hypersurfaces elliptic in the sense of \cite{TVV} would generalize to imply progress toward Conjecture~\ref{C:elliptic conj}, but here we claim results only in the bilinear range, wherein they are straight forward to deduce from results already in the literature \cite{Candy}.  (We will detail the deduction from the results of \cite{Candy} in Section~\ref{S:pos}.)

\begin{theorem}[\cite{Candy, Fefferman, TaoParab, TVV, Zygmund}]\label{T:elliptic thm}
Conjecture~\ref{C:elliptic conj} holds for all $d \geq 1$ and $q > \tfrac{2(d+3)}{d+1}$.  
\end{theorem}

As promised, we turn now to the dependence of operator norms on the sidelengths.  

\begin{conjecture} \label{C:rect conj}
Let $\ell \in (0,\infty]^d$ satisfy $l_1 \leq \cdots \leq l_d$.  For $g$ elliptic over $Q^\ell$ with parameters $N$ sufficiently large and $\eps_0>0$ sufficiently small, depending on $d,p,q$, we have the following operator norm bounds for $\scriptE^l_g$, with implicit constants independent of $g$ and $\ell$.  If $q>p$ satisfy $q=\tfrac{d-j-\theta+2}{d-j-\theta} p'$, for some $0 \leq j < d$ and $0 \leq \theta \leq 1$, then
\begin{equation} \label{E:q > p}
\|\scriptE^\ell_g\|_{L^p \to L^q} \lesssim (l_1 \cdots l_j l_{j+1}^\theta)^{\frac1{p'}-\frac1q};
\end{equation}
in particular, this quantity is finite whenever $l_{j+1} < \infty$.  If $l_d < \infty$, we have, in addition:  
\begin{align} \label{E:q leq p d>1}
\|\scriptE^\ell_g\|_{L^p \to L^q} \lesssim_{\eps} (l_1\cdots l_d)^{\frac1q-\frac1p}(\tfrac{l_d}{l_{j+1}})^\eps(l_1\cdots l_jl_{j+1}^\theta)^{1-\frac2q}, 
\end{align}
for $q = \tfrac{2(d-j-\theta+1)}{d-j-\theta} \leq p$, $0 < \theta \leq 1$, $\eps > 0$, and $j=0,\ldots,d-2$; and
\begin{align} \label{E:q > 4}
\|\scriptE^\ell_g\|_{L^p \to L^q} \lesssim (l_1 \cdots l_{d-1})^{\frac1{p'}-\frac1q}l_d^{1-\frac 3q-\frac1p},
\end{align}
for $q > 4$ and $p \geq (\tfrac q3)'$.  
\end{conjecture}

Modulo the precise definition of ellipticity, the two-dimensional version of this conjecture was essentially formulated by Buschenhenke--M\"uller--Vargas in \cite{BMVadd} (one must rescale).  

We have the following positive result.  

\begin{theorem} \label{T:pos}
Conjecture~\ref{C:elliptic conj} implies Conjecture~\ref{C:rect conj}.  In particular, Conjecture~\ref{C:rect conj} holds in the regions $q > \frac{10}3$ and also, when $d\geq 3$, for $(\tfrac1p,\tfrac1q)$ lying in the convex hulls of any of the following pairs of half-open line segments:
\begin{gather*}
\bigl[(1,0),(\tfrac{k^2+3k-2}{2k(k+3)},\tfrac{k+1}{2(k+3)})\bigr), \qtq{and} \bigl[(1,0),(\tfrac{(k+1)^2+3(k+1)-2}{2(k+1)(k+4)},\tfrac{k+2}{2(k+4)})\bigr), \quad 2 \leq k \leq d-1.
\end{gather*}
\end{theorem}

More precise statements of the conditionality may be found in the lemmas leading to the proof of Theorem~\ref{T:pos}.  

In addition, we prove that Conjecture~\ref{C:rect conj}, if true, is essentially optimal, excepting the precise asymptotics as $\frac{l_d}{l_{j+1}} \to \infty$ in the region $q \leq p$, $q \leq 4$.  

\begin{theorem} \label{T:neg}
Let $\ell \in (0,\infty]^d$ satisfy $l_1 \leq \cdots \leq l_d$.  Let $g$ be elliptic over $Q^\ell$ with parameters $N \geq 2$ and $0 < \eps_0 <1$.  Then $\scriptE^l_g$ does not extend as a bounded linear operator for $(p,q)$ lying outside of the region $q \geq \frac{d+2}d p'$, $q > \frac{2(d+1)}d$.  If $l_k = \infty$, some $1 \leq k \leq d$, then $\scriptE^l_g$ does not extend as a bounded operator from $L^p$ to $L^q$ for any $p \geq q$ nor $q \leq \frac{d-k+3}{d-k+1}p'$.  More precisely,  if $q>p$ satisfy $q=\tfrac{d-j-\theta+2}{d-j-\theta} p'$, for some $0 \leq j < d$ and $0 \leq \theta \leq 1$, then
\begin{equation} \label{E:q > p lb}
\|\scriptE^\ell_g\|_{L^p \to L^q} \gtrsim (l_1 \cdots l_j l_{j+1}^\theta)^{\frac1{p'}-\frac1q}.
\end{equation}
If $q = \tfrac{2(d-j-\theta+1)}{d-j-\theta} \leq p$, $0 < \theta \leq 1$, and $j\in\{0,\ldots,d-2\}$, then
\begin{align} \label{E:q leq p d>1}
\|\scriptE^\ell_g\|_{L^p \to L^q} \gtrsim (l_1\cdots l_d)^{\frac1q-\frac1p}\alpha(\tfrac{l_d}{l_{j+1}})(l_1\cdots l_jl_{j+1}^\theta)^{1-\frac2q}. 
\end{align}
Here $\alpha$ depends on $d$, $p$, and $q$; $\alpha \gtrsim 1$; and $\alpha(r) \to \infty$ as $r \to \infty$.  Finally, for $q > 4$ and $p \geq (\tfrac q3)'$,
\begin{align} \label{E:q leq p d=1}
\|\scriptE^\ell_g\|_{L^p \to L^q} \gtrsim (l_1 \cdots l_{d-1})^{\frac1{p'}-\frac1q}l_d^{1-\frac 3q-\frac1p}.
\end{align}
\end{theorem}

Correct attribution for the statement of Conjecture~\ref{C:rect conj} and prior progress toward Theorems~\ref{T:pos} and~\ref{T:neg} in the literature is somewhat ambiguous, particularly as some prior progress on these questions was not formalized into precisely stated theorems, the hypotheses and generality elsewhere differ, and the  implications of earlier methods and results seem not to have been fully exploited.  We give a recounting of the progress of which we are aware.  For the fully conditional part of Theorem~\ref{T:pos}, we use an elementary deduction, which was used to obtain an alternate proof of the restriction inequality for the cone in \cite{DruryGuo1993}.  In two dimensions, under a more restrictive hypothesis, lower bounds matching those from Theorem~\ref{T:neg} in the region $2p'\leq q\leq 3p'$, $p<q$ were obtained, as were the lower bounds in the region $3<q<4$, modulo the additional gain in $\frac{l_d}{l_{j+1}}$; it was also remarked that the methods in \cite{FU2009} (which, in turn, attributes the method to \cite{DruryGuo1993}) lead to the conditional result in this region.  Nevertheless, the question seems not to have been formulated in this level of generality (particularly with regard to dimension), some of our lower bounds seem to be new in all dimensions, and some of our positive results (in the bilinear range) are obtained by means that also seem to be new in this context.  

Two natural open questions are whether, for particular values of $\ell$, there is a larger range of exponents for which unconditional progress toward Theorem~\ref{T:pos} can be made, and whether unconditional results could be extended along horizontal lines in greater generality than just the bilinear range in two dimensions.

The analogues of Theorems~\ref{T:pos} and~\ref{T:neg} are obtained for the analogous convolution operators in an unpublished manuscript of the first author.  

Our main application is to determine new inequalities and give a simpler proof of known inequalities for a class of degenerate hypersurfaces.  Given $\beta \in (1,\infty)^d$, we define an extension operator
$$
\scriptE_\beta f(t,x) = \int_{Q^{\1}} e^{i(t,x)(g_\beta(\xi),\xi)}f(\xi)\, d\xi, \qquad g_\beta(\xi):=\sum_{j=1}^d |\xi_j|^{\beta_j}.
$$
In the case $d=2$, this extension operator was considered in \cite{FU2009} in the Stein--Tomas range and in \cite{BMVadd} in the bilinear range.  

Varchenko's height \cite{Varchenko} associated to these surfaces is the quantity $h$ defined by
$$
\tfrac1h:=\tfrac1{\beta_1}+\cdots+\tfrac1{\beta_d}.
$$
In determining bounds for $\scriptE_\beta$, intermediate dimensional versions of the height become relevant.  Thus, taking the convention that $\beta_1 \geq \beta_2 \geq \cdots \geq \beta_d$, we also define
$$
J_n:=\tfrac1{\beta_1}+\cdots+\tfrac1{\beta_n}, \qquad 0 \leq n \leq d.
$$

We obtain an essentially optimal conditional result for the operators $\scriptE_\beta$.  To facilitate its statement, we let $T_d$ denote the set of all $(p,q) \in [1,\infty]^2$ for which the local elliptic extension operator is conjectured to be bounded, that is,
\begin{equation}\label{E:def Td}
T_d := \{(p,q) \in [1,\infty]^2 : q > \tfrac{2(d+1)}d, \: q \geq \tfrac{d+2}d p'\}.
\end{equation}

\begin{theorem} \label{T:E_beta}
Assume that Conjecture~\ref{C:rect conj} holds for all $(\tilde p,\tilde q)$ in a relatively open subset $V \subseteq T_d$ containing $(p,q)$.  Then $\scriptE_\beta$ extends as a bounded operator from $L^p$ to $L^q$ if at least one of the following conditions hold:\\
\emph{(i)} $q>p$ and $\tfrac q{p'} \geq 1+\tfrac1{J_n+\frac{d-n}2}$, for all $0 \leq n \leq d$;\\
\emph{(ii)} $q \leq p$ and $\tfrac{1+J_n+d-n}q < \tfrac{J_n}{p'} + \tfrac{d-n}2$, for all $0 \leq n \leq d$; or\\
\emph{(iii)} $q=p$, $\tfrac{1+J_d}q=\tfrac{J_d}{p'}$, and $\tfrac{1+J_n+d-n}q < \tfrac{J_n}{p'} + \tfrac{d-n}2$, for all $0 \leq n < d$.\\
Furthermore, $\scriptE_\beta$ is of restricted weak type $(p,q)$ if \\
\emph{(iv)} $q \leq p \leq  \infty$, $\tfrac{1+J_d}q=\tfrac{J_d}{p'}$, and $\tfrac{1+J_n+d-n}q < \tfrac{J_n}{p'} + \tfrac{d-n}2$, for all $0\leq n<d$.
\end{theorem}

Here we use the not-completely-standard definition that a linear operator $T$, initially defined on $L^1(\R^d)$, is of restricted weak type $(p,q)$ if 
$$
|\langle Tf_E,g_F \rangle| \lesssim |E|^{\frac1p}|F|^{\frac1{q'}},
$$
for all measurable, finite measure $E,F$ and measurable functions $|f_E| \leq \chi_E$ and $|g_F| \leq \chi_F$.  (It will be convenient to note that we may equivalently replace `$\leq$' by `$\sim$' in the conditions on $f_E,g_F$.)  For finite $p$, this is equivalent to the usual definition of restricted weak type boundedness.  

Conditional on the restriction conjecture above rectangles, both the strong and restricted weak type estimates arising in Theorem~\ref{T:E_beta} are sharp.

\begin{proposition} \label{P:sharp E_beta}
If $(p,q) \in [1,\infty]^2$ does not satisfy any of Conditions \emph{(i-iii)} of Theorem~\ref{T:E_beta}, then $\scriptE_\beta$ is not of strong type $(p,q)$.  If $(p,q) \in [1,\infty]^2$ does not satisfy any of Conditions \emph{(i-iv)} of Theorem~\ref{T:E_beta}, then $\scriptE_\beta$ is not of restricted weak type $(p,q)$. 
\end{proposition}

When $d=2$, Proposition~\ref{P:sharp E_beta} is due to \cite{BMVadd}, and Theorem~\ref{T:E_beta} is due to \cite{BMVadd} in the bilinear range (where it is unconditional) and to \cite{FU2009} in the Stein--Tomas range.  For $d>2$, Theorem~\ref{T:E_beta} is due to \cite{FU2008} in the Stein-Tomas range. Our main new contribution is a direct deduction of the result from Conjecture~\ref{C:elliptic conj}, which leads to a simpler approach that avoids the complicated step of obtaining bilinear restriction estimates between rectangles at different scales.  This simplification enables us to address the higher dimensional case, as well as the case when some exponents $\beta_i$ are less than 2.

The region $(\tfrac1p,\tfrac1q)$ described in Theorem~\ref{T:E_beta} and Proposition~\ref{P:sharp E_beta} can be somewhat difficult to visualize, so we make a few simple observations.  We see the familiar conditions $q \geq \tfrac{d+2}d p'$ and $q>\tfrac{2(d+1)}d$ in the $n=0$ case of each of the constraints.  The lower bound on $\tfrac q{p'}$ in (i) of Theorem~\ref{T:E_beta} is strongest when $J_n+\frac{d-n}{2}$ is minimal, which occurs when $n=n_0$, the minimal index for which $\beta_i<2$ for all $i > n_0$.  Thus the constraint in (i) is strictly stronger than that in the elliptic restriction conjecture unless $n_0=0$, i.e.\ $\beta_i \leq 2$ for all $i$.  

The condition (ii) may introduce some vertices in the Riesz diagram; these all lie on or above the line $\tfrac1q=\tfrac1p$.  For each $n$, the lines $\frac{q}{p'}= 1+\frac{1}{J_n+\frac{d-n}{2}}$ (seen in (i)) and $\tfrac{1+J_n+d-n}q=\tfrac{J_n}{p'}+\tfrac{d-n}{2}$ (seen in (ii)) intersect when $q=p=2+\frac{1}{J_n+\frac{d-n}{2}}$, and these two lines are equal when $n=d$.  The slope of the line $\tfrac{1+J_n+d-n}q=\tfrac{J_n}{p'}+\tfrac{d-n}{2}$ is $-\tfrac{J_n}{1+\tfrac{1+d-n}{J_n}}$, which equals 0 when $n=0$ and decreases as $n$ increases.  
The intersection point of such a line with $q=p$ is $q=p=2+\frac{1}{J_n+\frac{d-n}{2}}$, which moves closer to $(0,0)$ as $n$ increases until $n$ reaches $n_0$, at which point it begins to increase.  Thus only those lines with $n \leq n_0$ play a role in determining the boundary of the region.

\subsection*{Notation} Admissible constants may depend on the dimension $d$, the exponents $p,q$, and the $d$-tuple $\beta$ in the definition of $\scriptE_\beta$, as well as any operator norms on whose finiteness results may be conditioned.  For nonnegative real numbers $A,B$, we will use the notation $A \lesssim B$, $B \gtrsim A$ to mean that $A \leq C B$ for an admissible constant $C$, which is allowed to change from line to line; $A \sim B$ means $A \lesssim B$ and $B \lesssim A$.  We will occasionally subscript constants or the $\lesssim$ notation to indicate dependence on an additional parameter.  For $\lambda \in \R$, $\lambda_+$ denotes the positive part, $\lambda_+ :=\max\{\lambda,0\}$.  

\subsection*{Acknowledgements}  This research was supported in part by National Science Foundation grants DMS-1653264 and DMS-1147523.

\section{Dicing, slicing, and the Morse Lemma on elliptic surfaces}

In this section we will prove a handful of technical lemmas that are useful in developing the restriction theory of hypersurfaces elliptic over a rectangle.  The proofs use only basic calculus.  

The statements of the results will be simpler if we generalize the notion of ellipticity.  

\begin{definition} \label{D:elliptic cvx}
Let $K$ be a convex subset of $\R^d$ with nonempty interior, and let $g \in C^{N+2}_{\rm{loc}}(K)$, with $D^2 g$ positive definite throughout $K$.  We say that $g$ is elliptic over $K$, with parameters $N,\eps_0$, if there exist $\xi_0 \in K$, $U \in O(d)$, and $\ell \in (0,\infty]^d$ such that
$$
K \subseteq \sqrt 2 [D^2 g(\xi_0)]^{-\frac12} U Q^\ell + \xi_0,
$$
and the functions 
\begin{equation} \label{E:tilde g}
\tilde g(\xi):=g(\sqrt 2 D^2 g(\xi_0)^{-\frac12} U \xi + \xi_0) - g(\xi_0) - \sqrt 2 D^2 g(\xi_0)^{-\frac12} U \xi \cdot \nabla g(\xi_0)
\end{equation}
and $\tilde h(\xi):=\tilde g(\xi)-|\xi|^2$ obey
\begin{equation} \label{E:elliptic cvx}
\|\ell^\alpha \partial^\alpha D^2\tilde h\|_{C^N(\tilde K)} < \eps_0, 
\end{equation}
where 
\begin{equation} \label{E:tilde K}
\tilde K:= \tfrac1{\sqrt 2}D^2 g(\xi_0)^{\frac12}U^T(K-\xi_0)).
\end{equation}
\end{definition}

\subsection{Dicing}

Here we will prove that the restriction of a function elliptic over a rectangle is elliptic over smaller rectangles, with improved parameters.  Such results are extremely useful in induction on scales arguments and also in the proofs of our negative results.  

\begin{lemma} \label{L:dicing}
Let $\ell \in (0,\infty]^d$ and let $g$ be elliptic over $Q^\ell$ with parameters $(N,\eps_0)$, some $N \geq 1$.  Let $K \subseteq Q^\ell$ be a convex set with nonempty interior, and assume that $\eps^{-1}(K-\xi_0) \subseteq Q^\ell$ for some $\xi_0 \in K$ and $0 < \eps \leq 1$.  Then $g$ is elliptic over $K$ with parameters $N,C_{N,d} \eps \eps_0$.  
\end{lemma}

\begin{proof}[Proof of dicing lemma]
By taking limits, we may assume that $K$ is compact.  By the John ellipsoid theorem, there exists $\tilde\ell \in (0,\infty)^d$ and $U \in O(d)$ such that
$$
c_d U Q^{\tilde \ell} \subseteq \tfrac1{\sqrt 2} D^2 g(\xi_0)^{\frac12}(K-\xi_0) \subseteq UQ^{\tilde \ell}.
$$
Define $\tilde g$ as in \eqref{E:tilde g} and $\tilde K$ as in \eqref{E:tilde K}.  For $|\alpha| \geq 1$, 
\begin{align*}
\|{\tilde\ell}^\alpha \partial^\alpha D^2 \tilde g\|_{C^0(\tilde K)} &= \|2U^T D^2 g(\xi_0^{-\frac12}(\sqrt 2 D^2 g(\xi_0)^{-\frac12}U\tilde\ell)^\alpha \partial^\alpha D^2 g D^2 g(\xi_0)^{-\frac12}U\|_{C^0(K)}\\
& \leq 2\|D^2 g(\xi_0)^{-\frac12}\|^2\|\theta^\alpha \partial^\alpha D^2 g\|_{C^0(Q^\ell)},
\end{align*}
where 
$$
\theta:=\sqrt 2 D^2 g(\xi_0)^{-\frac12}U \tilde\ell \in K-\xi_0 \subseteq c_d^{-1}\eps Q^\ell.
$$
As $\|D^2 g(\xi_0)^{-\frac12}\| \leq C_d$, in the case $|\alpha| \geq 1$, \eqref{E:elliptic cvx} follows from the ellipticity hypothesis.  In the case $\alpha = 0$, \eqref{E:elliptic cvx} follows from the $|\alpha|=1$ case, $D^2 \tilde g(0) = 2\mathbb I_d$, and the fundamental theorem of calculus.  
\end{proof}

\subsection{Slicing}

Here we will show that the restriction of a function elliptic over a rectangle to some lower-dimensional slice of the rectangle is also elliptic, with comparable parameters.  This result is essential in the (conditional) proof of Theorem~\ref{T:pos}.  

\begin{lemma} \label{L:slicing}
Let $\ell \in (0,\infty]^d$ and let $g$ be elliptic over $Q^\ell$ with parameters $N,\eps_0$.  Let $P \subseteq \R^d$ be an affine $k$-plane, and assume that $P \cap Q^\ell$ has nonempty interior in $P$.  Let $\xi_0 \in P \cap Q^\ell$, and let $\{u_1,\cdots,u_k\}$ be an orthonormal basis for $P-\xi_0$.  Then
$$
 g^\flat(\eta_1,\ldots,\eta_k) :=g(\xi_0+\sum_{j=1}^k \eta_j u_j)
$$
is elliptic over $K:=\{(\eta_1,\ldots,\eta_k):\xi_0+\sum_{j=1}^k \eta_j u_j \in P \cap Q^\ell\}$.
\end{lemma}

\begin{proof}  By taking limits, we may assume that $K$ is compact.  By the John ellipsoid theorem, there exists $\tilde \ell \in (0,\infty)^k$ and $V \in O(k)$ such that
$$
c_d V Q^{\tilde\ell} \subseteq \tfrac1{\sqrt 2} D^2 g^\flat(0)^{\frac12}K \subseteq VQ^{\tilde\ell}.
$$
Set $\tilde K:=\tfrac1{\sqrt 2} V^TD^2 g^\flat(0)^{\frac12}K$, 
$$
\tilde g^\flat(\eta):=g^\flat(\sqrt2 D^2 g^\flat(0)^{-\frac12}V\eta)-g^\flat(0) - \sqrt 2 D^2 g^\flat(0)^{-\frac12}V\eta \cdot \nabla g^\flat(0),
$$
and $\tilde h^\flat(\eta) := \tilde g^\flat(\eta) - |\eta|^2$, $\eta \in \R^k$.  

Extend the given basis for $P-\xi_0$ to an orthonormal basis $\{u_1,\cdots,u_d\}$ of $\R^d$ and set $U:=(u_1,\ldots,u_d) \in O(d)$.  Then
$$
c_d U[\sqrt 2 D^2 \tilde g(0)^{-\frac12}VQ^{\tilde \ell} \times \{0\}] \subseteq (P \cap Q^\ell)-\xi_0 \subseteq U[\sqrt 2 D^2 \tilde g(0)^{-\frac12}VQ^{\tilde \ell} \times \{0\}].  
$$
Let $|\alpha| \geq 1$ be a multiindex.  By the chain rule,
\begin{align*}
\|\tilde\ell^\alpha \partial^\alpha D^2 \tilde g^\flat\|_{C^0(\tilde K)}
&\leq \|\sqrt 2 D^2 g^\flat(0)^{-\frac12}\|^2\|(\sqrt 2 D^2 g(0)^{-\frac12}V\tilde\ell)^\alpha \partial^\alpha D^2 g^\flat\|_{C^0(K)}\\
&\leq \|\sqrt 2 D^2 g^\flat(0)^{-\frac12}\|^2 \|[U(\sqrt 2 D^2 \tilde g(0)^{-\frac12}V\tilde\ell,0)]^\alpha \partial^\alpha D^2 g\|_{C^0(Q^\ell)}.
\end{align*}
Finally, since 
$$
U(\sqrt 2 D^2 \tilde g(0)^{-\frac12}V\tilde\theta,0) \in U[\sqrt 2 D^2 \tilde g(0)^{-\frac12}VQ^{\tilde \ell} \times \{0\}] \subseteq c_d^{-1}((P \cap Q^\ell)-\xi_0) \subseteq c_d^{-1}Q^\ell,
$$
inequality \eqref{E:elliptic cvx} holds for $|\alpha| \geq 1$; the case $|\alpha|=0$ follows analogously, by considering $\tilde h^\flat$.  
\end{proof}
 
 \subsection{Diffeomorphizing}
 
 Finally, we note that the following version of the Morse lemma follows readily by adapting standard undergraduate-level proofs of the Morse lemma to functions elliptic over rectangles.  (We omit the details of this elementary adaptation.)  This result allows us to invoke the classical arguments involving stationary phase, including the Stein--Tomas and Strichartz theorems.   

\begin{lemma} \label{L:Morse}
Let $g$ be elliptic over $Q^\ell$, $\ell \in (0,\infty)^d$, with parameters $N \geq 1$ and $0 < \eps \leq \eps_d$.  Then exist $U \subseteq \R^d$ and a $C^N$ diffeomorphism $F$ of $U$ onto $Q^\1$ such that
$$
g (l_1 F_1,\ldots,l_d F_d) = \sum_j (l_j u_j)^2
$$
and $\|F(\eta)-\eta\|_{C^N_\eta(U)} < C(\eps)$.  
\end{lemma}

\section{Negative results:  The proof of Theorem~\ref{T:neg}}

For simplicity, we give a complete proof of Theorem~\ref{T:neg}, recalling that it is already known in some cases.  We will actually prove a slightly stronger result.  Let 
$$
\|\scriptE^\ell_g\|_{L^p \to L^q}^{\rm{RWT}}:=\sup_{E,F} \tfrac{|\langle \scriptE^\ell_g f_E,g_F \rangle|}{|E|^{\frac1p}|F|^{\frac1{q'}}},
$$
where the supremum is taken over measurable sets $E,F$ with positive, finite measures and measurable functions $|f_E| \leq \chi_E$, $|g_F| \leq \chi_F$.  

\begin{proposition} \label{P:neg}
The conclusions of Theorem~\ref{T:neg} hold with $\|\scriptE^\ell_g\|_{L^p \to L^q}^{\rm{RWT}}$ in place of $\|\scriptE^\ell_g\|_{L^p \to L^q}$.  
\end{proposition}

The rest of this section will be devoted to the proof of Proposition~\ref{P:neg}.  We will use the {convention} that references to equations in the statement of Theorem~\ref{T:neg} shall be superscripted with `$\rm{RWT}$.'

That $\|\scriptE^\ell_g\|_{L^p \to L^q}^{\rm{RWT}}$ is infinite whenever $q<\tfrac{d+2}d p'$ follows from the classical Knapp example.  The case when $q \leq \tfrac{2(d+1)}d$ follows from a slight modification of the argument from \cite{BCSS}, which will be given in Lemma~\ref{L:Kakeya}.  The assertion regarding unboundedness of $\scriptE^\ell_g$ when some $l_j$ are infinite follows from the lower bounds (\ref{E:q > p lb}-\ref{E:q leq p d=1}) and the elemenary inequality 
\begin{equation} \label{E:monotonicity}
\|\scriptE^\ell_g\|_{L^q \to L^p}^{\rm{RWT}} \geq \|\scriptE^{\tilde\ell}_g\|_{L^q \to L^p}^{\rm{RWT}}, \qtq{for} \tilde l_j \leq l_j,\: j=1,\ldots,d.  
\end{equation}

It remains to prove the lower bounds in the case that each $l_j$ is finite.  This will be carried out in three lemmas, one for each numbered inequality.

\begin{lemma}\label{L:q>p neg}
The lower bound \eqref{E:q > p lb} from Theorem~\ref{T:neg} is valid in the range $q>p$.
\end{lemma}

\begin{proof}
The argument is an elementary generalization of the Knapp example.  Let $j\in\{0,\ldots,d-1\}$, $0 \leq \theta \leq 1$, and assume that $q>p$ satisfy $q=\tfrac{d-j-\theta+2}{d-j-\theta}p'$.  

Let $\phi$ be a smooth, nonnegative function with $\supp \phi \subseteq Q^\1$ and $\int\phi=1$.  Set
$$
\phi^{j,\ell}(\xi) = \phi(\tfrac{\xi_1}{l_1},\ldots,\tfrac{\xi_j}{l_j},\tfrac{\xi_{j+1}}{l_{j+1}},\ldots,\tfrac{\xi_d}{l_{j+1}}).
$$
Then $|\scriptE^\ell_g \phi^{j,\ell}| \gtrsim l_1\cdots l_j l_{j+1}^{d-j}$ on a rectangle with volume $(l_1\cdots l_j)^{-1}l_{j+1}^{-(d-j+2)}$.  After a little arithmetic
\begin{equation} \label{E:q>p neg}
\|\scriptE^\ell_g\|_{L^p \to L^q}^{\rm{RWT}} \gtrsim \tfrac{l_1\cdots l_j l_{j+1}^{d-j} [(l_1\cdots l_j)^{-1}l_{j+1}^{-(d-j+2)}]^{\frac1q}}{(l_1 \cdots l_j l_{j+1}^{d-j})^{\frac1p}} = (l_1\cdots l_j)^{\frac1{p'}-\frac1q}l_{j+1}^{(d-j)(\frac1{p'}-\frac1q)-\frac2q}.
\end{equation}
All that remains is the arithmetic verification of $(d-j)(\frac1{p'}-\frac1q)-\frac2q=\theta(\frac1{p'}-\frac1q)$ when $q=\frac{d-j-\theta+2}{d-j-\theta} p'$; we leave this to the reader.
\end{proof}

\begin{lemma}
The lower bound \eqref{E:q leq p d=1} from Theorem~\ref{T:neg} is valid in the range $q \leq p$, $q>4$.  
\end{lemma}

\begin{proof}
We assume $q>4$ and $p \geq q$.  We use the same $\phi^\ell$ from the proof of Lemma~\ref{L:q>p neg}, and inequality \eqref{E:q>p neg}, which can be rearranged into \eqref{E:q leq p d=1}, remains valid.
\end{proof}

\begin{lemma} \label{L:Kakeya}
The lower bound \eqref{E:q leq p d>1} from Theorem~\ref{T:neg} is valid in the range $q \leq p$, $q \leq 4$.  Moreover, $\|\scriptE^\ell_g\|_{L^p \to L^q}^{\rm{RWT}} = \infty$ whenever $q \leq \tfrac{2(d+1)}d$.  
\end{lemma}

\begin{proof}
We begin with the case $q=\tfrac{2(d-j-\theta+1)}{d-j-\theta}$ and $p \geq q$ for some $0 \leq j \leq d-2$ and $0 < \theta \leq 1$.    We will argue by adapting the Kakeya-like argument of \cite{BCSS}.  

Let $N \gg 1$.  Assume that $\frac{l_d}{l_{j+1}} > N^3$.  By parabolic rescaling, we may assume that $l_{j+1} = \tfrac1N$.  We cover (at least half of) $Q^\ell$ by pairwise disjoint rectangles $R \in \scriptR$ congruent to $Q^{(l_1,\ldots,l_j,1/N,\ldots,1/N,1)}$ and further decompose each $R$ into a disjoint union of rectangles $\kappa \in \scriptK_R$, each congruent to $Q^{(l_1,\ldots,l_j,1/N,\ldots,1/N)}$.  By the usual Knapp argument,
$$
|\scriptE^\ell_g \chi_\kappa(t,x)| \gtrsim |\kappa|
$$
on a tube $T_\kappa$ of volume $|T_\kappa| \sim |\kappa|^{-1} l_{j+1}^{-2}$.  We will prove that for each $R$, there exist $\{(t_\kappa,x_\kappa)\}_{\kappa \in \scriptK_R}$ such that
\begin{equation} \label{E:Kakeya}
|\bigcup_{\kappa \in \scriptK_R} T_\kappa+(t_\kappa,x_\kappa)| < o_N(1)\#\scriptK_R|T_\kappa|, 
\end{equation}
where $o_N(1) \to 0$ as $N \to \infty$ and is otherwise independent of $\ell$.  

Define 
$$
F:=\sum_{R \in \scriptR} e^{-i(t_R,x_R)(g(\xi),\xi)}\sum_{\kappa \in \scriptK_R} \omega_\kappa e^{-i(t_\kappa,x_\kappa)(g(\xi),\xi)}\chi_\kappa,
$$
with $\{(t_R,x_R)\}_{R \in \scriptR} \subseteq \R^{1+d}$ and $\{\omega_R\}_{\kappa \in \scriptK_R} \subseteq \{\pm1\}$ to be determined shortly.  Of course, $|F| \lesssim \chi_{Q^\ell}$.  For the $(t_R,x_R)$ sufficiently widely spaced (depending on the $(t_\kappa,x_\kappa)$), the $L^{q,\infty}$ norms decouple:  
$$
\|\scriptE^\ell_g F\|_{L^{q,\infty}} \gtrsim \bigl(\sum_{R \in \scriptR} \|\scriptE^\ell_g F_R\|_{L^{q,\infty}}^q\bigr)^{\frac1q}, \qquad F_R:=\sum_{\kappa \in \scriptK_R} \omega_\kappa e^{-i(t_\kappa,x_\kappa)(g(\xi),\xi)}\chi_\kappa.
$$

By Khintchine's inequality, we may choose the $\omega_\kappa$ such that
$$
\|\scriptE^\ell_g F_R\|_{L^{q,\infty}} \gtrsim \|\bigl(\sum_{\kappa \in \scriptK_R} |\scriptE^\ell_g F_\kappa|^2\bigr)^{\frac12}\|_{L^{q,\infty}}, \qquad F_\kappa:= e^{-i(t_\kappa,x_\kappa)(g(\xi),\xi)}\chi_\kappa.
$$
Applying our pointwise lower bound on $\scriptE^\ell_g \chi_\kappa$, H\"older's inequality, and \eqref{E:Kakeya},
\begin{align*}
&\|\bigl(\sum_{\kappa \in \scriptK_R} |\scriptE^\ell_g F_\kappa|^2\bigr)^{\frac12}\|_{L^{q,\infty}} 
\gtrsim |\kappa|\|\sum_{\kappa \in \scriptK_R} \chi_{T_\kappa+(t_\kappa,x_\kappa)}\|_{L^{\frac q2,\infty}}^{\frac12}\\
&\qquad \gtrsim |\kappa||\bigcup_{\kappa \in \scriptK_R} T_\kappa+(t_\kappa,x_\kappa)|^{\frac1q-\frac12}\|\sum_\kappa \chi_{T_\kappa+(t_\kappa,x_\kappa)}\|_{L^1}^{\frac12}
\gtrsim |\kappa| o_N(1)^{-1} (\#\scriptK_R|T_\kappa|)^{\frac1q}.
\end{align*}
It takes a little arithmetic to put the pieces together:
\begin{align*}
\|\scriptE^\ell_g\|_{L^p \to L^q}^{\rm{RWT}} &\gtrsim \frac{\|\scriptE^\ell_g F\|_{q,\infty}}{|Q^\ell|^{\frac1p}} \gtrsim |Q^\ell|^{-\frac1p}|\kappa| o_N(1)^{-1}(|Q^\ell| |\kappa|^{-2} \ell_{j+1}^{-2})^{\frac1q}\\
&\sim o_N(1)^{-1}|Q^\ell|^{\frac1q-\frac1p} (l_1 \cdots l_j l_{j+1}^\theta)^{1-\frac2q}.
\end{align*}
Thus the lemma is proved modulo the Kakeya-like inequality \eqref{E:Kakeya}.  

For $\xi \in R$, we use Taylor's theorem to estimate
\begin{align*}
g(\xi)&= g(\xi_R) + (\xi-\xi_R)\cdot\nabla g(\xi_R) + \tfrac12 (\xi-\xi_R)^T D^2 g(\xi_R) (\xi-\xi_R)\\
&\qquad + O(\sum_{|\alpha|=3} |(\xi-\xi_R)^\alpha|\|\partial^\alpha g\|_{C^0(Q^\ell)}),
\end{align*}
where $\xi_R$ denotes the center of $R$.  We examine the error term.  If $\alpha_d=0$ and $\alpha_i \neq 0$,
$$
|(\xi-\xi_R)^\alpha|\|\partial^\alpha g\|_{C^0(Q^\ell)} \leq \tfrac1{N^2} \ell_i \|\partial_i D^2 g\|_{C^0(Q^\ell)} < \tfrac{\eps_0}{N^2},
$$
since $|(\xi-\xi_R)_k| < \min\{\frac1N,l_k\}$, $1 \leq k < d$.  If $\alpha_d\neq 0$,
$$
|(\xi-\xi_R)^\alpha|\|\partial^\alpha g\|_{C^0(Q^\ell)} \leq \tfrac1{N^2}l_d\|\partial_d D^2 g\|_{C^0(Q^\ell)} < \tfrac{\eps_0}{N^2},
$$
since $|\xi-\xi_R| \leq 1 < \tfrac1{N^2}l_d$.  Take $\xi \in \kappa \subseteq R$ and let $\xi_\kappa$ denote the center of $\kappa$.  From the preceding and the definition of $\kappa$, for $\xi \in \kappa$,
$$
g(\xi) = (\xi-\xi_\kappa) \cdot (\nabla g(\xi_R) + D^2g(\xi_R)(\xi_\kappa-\xi_R)) + c(\kappa) + O(\tfrac1{N^2}),
$$
where $c(\kappa)$ is independent of $\xi$.  

By construction, $\xi_\kappa - \xi_R = \tfrac{n_\kappa}N e_d$, some $n_\kappa \in \{-(N-1),\ldots,0,\ldots,N-1\}$.  Thus for $|t| < cN^2$, $c$ sufficiently small,
$$
(t,x)(g(\xi)-c(\kappa),\xi-\xi_\kappa) = c \, O(1) + (\xi-\xi_\kappa)(t\nabla g(\xi_R) + t\tfrac{n_\kappa}N \partial_d \nabla g(\xi_R) + x).
$$
We therefore see that
$$
|\scriptE_g^\ell \chi_\kappa(t,x)| \gtrsim |\kappa|
$$
on
$$
T_\kappa:=\{(t,x): |t|<cN^2, \: |t\partial_i g(\xi_R) + t\tfrac{n_\kappa}N \partial_d \partial_i g(\xi_R) + x_i| < c L_i\}
$$
where $L_i =  \ell_i^{-1}, \qtq{if} i \leq j$, and $L_i = N$ if $i > j$.

The linear transformation
$$
A_R(t,s,y) := (t,-t\nabla g(\xi_R) - s\partial_d \nabla g(\xi_R)+(y,0)), \qquad (t,s,y) \in \R^{1+1+(d-1)}
$$
has determinant $|\det A_R| =  \partial_d^2g(\xi_R) \sim 1$ and maps
$$
T_\kappa^\flat:=\{(t,s,y): |t|<cN^2, \: |s-t\tfrac{n_\kappa}{N}| < cN, |y_i|<cL_i\} 
$$
onto $T_\kappa$, recalling that constants are allowed to change from line to line.  Using (e.g.) Fefferman's construction \cite{ball_non_mult}, there exist $\{(t_\kappa,s_\kappa)\}_{\kappa\in\scriptK_R}$ such that
$$
|\bigcup_{\kappa \in \scriptK_R} T_\kappa^\flat + (t_\kappa,s_\kappa,0)| < o_N(1)\sum_{\kappa \in \scriptK_R}|T_\kappa^\flat|.
$$
Finally, $(t_\kappa,x_\kappa)$ with
$$
x_\kappa:=-t_\kappa \nabla g(\xi_R) - s_\kappa \partial_d \nabla g(\xi_R)
$$
gives \eqref{E:Kakeya}.  

The proof  of the necessity of $q > \tfrac{2(d+1)}d$ is similar.  By rescaling and using monotonicity in $\ell$ of the operator norms, we may assume that $\ell = \1$.  We cover $Q^\1$ by rectangles congruent to $Q^{(1/N,\ldots,1/N,1)}$ and cover these by smaller cubes $\kappa$ congruent to $Q^{(1/N,\ldots,1/N)}$.  Following the argument above, 
$$
\|\scriptE_g^\1\|_{L^p \to L^q}^{\rm{RWT}} \gtrsim o_N(1)^{-1}N^{\frac{2(d+1)}q-d},
$$
which is unbounded as $N \to \infty$ for any $q \leq \frac{2(d+1)}d$.  
\end{proof}

\section{Upper bounds above rectangles} \label{S:pos}

In this section, we provide details for the deduction of Theorem~\ref{T:elliptic thm} from known results and prove Theorem~\ref{T:pos}. By Fatou's lemma, it suffices to prove these results when $\ell \in (0,\infty)^d$.  We recall the hypothesis that $l_1 \leq \cdots \leq l_d$.  

We begin by collecting from the literature the details needed to establish Theorem~\ref{T:elliptic thm}.  We will apply Theorem~1.4 of \cite{Candy} to obtain a bilinear restriction result.  For the convenience of the reader, we record that, in the notation of \cite{Candy}, we are using: $q=r>\tfrac{d+3}{d+1}$, $\scriptH_1 = \scriptH_2 = \scriptV_{max} = 1$, and $\mathbf{d_0}$ is the ball of radius $\min\{1,l_1\}$, $\mathbf d[\cdot,\cdot]$ bounds the measures of certain ``intersection hypersurfaces'' in $\R^d$ (see (1.10)).  

\begin{lemma}[\cite{Candy}] \label{L:Candy}
Let $g$ be elliptic over the rectangle $Q^\ell$, with $l_d \geq 1$, and let $B_1,B_2$ be two balls of radius 1, separated by a distance 1 and intersecting $Q^\ell$.  For functions $f_j \in L^2(Q^\ell)$, $\supp f_j \subseteq B_j$ and $q > \tfrac{2(d+3)}{d+1}$,
$$
\|\scriptE^\ell_g f_1\,\scriptE^\ell_g f_2\|_{\frac q2} \lesssim \|f_1\|_2\|f_2\|_2.  
$$
\end{lemma}

Using the invariance of the ellipticity hypothesis under parabolic rescaling, we obtain a bilinear restriction estimate for functions supported on balls of radius $r$, separated by a distance $r$, which leads to Theorem~\ref{T:elliptic thm} via the bilinear-to-linear argument of Tao--Vargas--Vega \cite{TVV}.  We omit the details.

We now turn to the proof of Theorem~\ref{T:pos}, beginning with bounds corresponding directly to lower-dimensional restriction theorems.  

\begin{lemma}[\cite{DruryGuo1993}] \label{L:DG}
Theorem~\ref{T:pos} holds on each half-open segment $q = \tfrac{d-j+2}{d-j}p'$, $p < q$, $j=1,\ldots,d$. In particular \eqref{E:q > p} holds unconditionally on the segment $q=3p'$, $q > p$. Moreover, validity of Conjecture~\ref{C:elliptic conj} in $d-j$ dimensions with exponents $(p,q)$ implies Conjecture~\ref{C:rect conj} in dimension $d$ with the same exponent pair.  
\end{lemma}

\begin{proof}
Because the lemma was not stated as such in \cite{DruryGuo1993}, we give the complete proof.  When $p=1,q=\infty$, the result is a direct application of H\"older's inequality.  Let $0 \leq j \leq d-1$, let $q=\tfrac{d-j+2}{d-j} p' > p > 1$, and assume that Conjecture~\ref{C:elliptic conj} is valid in dimension $d-j$ for this exponent pair.  Given $f \in C^\infty_0(Q^\ell)$, we take the Fourier transform of $\scriptE^\ell_g$ in the $x'$ variables to obtain,
$$
\scriptF_{x'} \scriptE^\ell_gf(t,x'')(\xi') = \scriptE^{\ell''}_{g_{\xi'}}f_{\xi'}(t,x''),
$$
where we have split the coordinates as $x=(x',x'') \in \R^j\times\R^{d-j}$, and we are writing $h_{\xi'}(\xi'') = h(\xi)$, for $h$ a function on $\R^d$.  After making a linear transformation, which amounts to replacing $g_{\xi'}(\xi'')$ with 
$$
g_{\xi'}(\xi'')-g_{\xi'}(0'')-\xi'' \cdot \nabla''g_{\xi'}(0''),
$$
we see that the hypothesized Conjecture~\ref{C:elliptic conj} applies uniformly to $\scriptE^{\ell''}_{g_{\xi'}}$, $\xi' \in Q^{\ell'}$.  

Now applying Hausdorff--Young (using $q>2$), then Minkowski's inequality (using $q'<q$), then the hypothesized validity of Conjecture~\ref{C:elliptic conj}, and finally H\"older's inequality (using $p > q'$),
\begin{align*}
\|\scriptE^\ell_g f\|_q &\leq \|\scriptE^{\ell''}_{g_{\xi'}}f_{\xi'}\|_{L^q_{t,x}(L^{q'}_{\xi'})} \leq \|\scriptE^{\ell''}_{g_{\xi'}}f_{\xi'}\|_{L^{q'}_{\xi'}(L^{q}_{t,x''})}  \lesssim \|f\|_{L^{q'}_{\xi'}(L^p_{\xi''})}\\
& \leq |Q^{\ell'}|^{\frac1{q'}-\frac1p}\|f\|_p = (l_1\cdots l_j)^{\frac1{p'}-\frac1q}\|f\|_p.
\end{align*}
This inequality is equivalent to \eqref{E:q > p} in the case $\theta=0$ (and in the case $j-1, \theta=1$).  
\end{proof}

\begin{lemma} \label{L:interstices}
Theorem~\ref{T:pos} holds on the region $\frac{d+2}d p' \leq q \leq 3p'$, $q>p$, in the sense that validity of Conjecture~\ref{C:elliptic conj} in $d-j-i$ dimensions with exponents $(p_i,q_i)$, $i=0,1$, implies validity of Conjecture~\ref{C:rect conj} in $d$ dimensions, for $(p^{-1},q^{-1})$ on the line segment joining $(p_0^{-1},q_0^{-1})$ and $(p_1^{-1},q_1^{-1})$.  
\end{lemma}

\begin{proof}
The argument is by the obvious interpolation.  By our hypothesis and Lemma~\ref{L:DG}, 
$$
\|\scriptE^\ell_g\|_{L^{p_i} \to L^{q_i}} \lesssim (l_1 \cdots l_j l_{j+1}^i)^{\frac1{p_i'}-\frac1{q_i}}, \qquad i=0,1.
$$
Setting 
$$
\tfrac1{p_\theta}:=\tfrac{1-\theta}{p_0}+\tfrac\theta{p_1}, \qquad \tfrac1{q_\theta} :=\tfrac{1-\theta}{q_0}+\tfrac\theta{q_1},
$$
as usual, interpolation gives
$$
\|\scriptE^\ell_g\|_{L^{p_\theta} \to L^{q_\theta}} \lesssim (l_1\cdots l_j)^{\frac1{p_\theta'}-\frac1{q_\theta}}l_{j+1}^{\frac\theta{p_1'}-\frac\theta{q_1}}.
$$
Inequality \eqref{E:q > p} in the claimed region thus follows once we prove that the equation
\begin{equation} \label{E:nu ok?}
\nu(\tfrac1{p_\theta'}-\tfrac1{q_\theta}) = \theta(\tfrac1{p_1'}-\tfrac1{q_1})
\end{equation}
is valid for the quantity $\nu=\nu_\theta$ defined implicitly by
$$
q_\theta =: \tfrac{d-j-\nu+2}{d-j-\nu} p_\theta'.
$$
(In other words, $\nu$ is the `$\theta$' from \eqref{E:q > p}.)  
Indeed, taking the convex combination of the scaling equations for $(p_0,q_0)$ and $(p_1,q_1)$ yields
\begin{equation} \label{E:nu ok1}
\tfrac{d-j}{p_\theta'} - \tfrac\theta{p_1'} = \tfrac{d-j+2}{q_\theta}-\tfrac\theta{q_1},
\end{equation} 
while the definition of $\nu$ can be rearranged as 
\begin{equation} \label{E:nu ok2}
\tfrac{d-j}{p_\theta'} - \tfrac\nu{p_\theta'} = \tfrac{d-j+2}{q_\theta}-\tfrac\nu{q_\theta},
\end{equation}
Subtracting \eqref{E:nu ok2} from \eqref{E:nu ok1} and rearranging yields \eqref{E:nu ok?}.  
\end{proof}

\begin{lemma} \label{L:q leq p}
In the region $\tfrac{2(d+1)}{d} < q\leq 4$, $q \leq p$, validity of Conjecture~\ref{C:elliptic conj} in dimensions $d-j-1$ and $d-j$ implies validity of Conjecture~\ref{C:rect conj} on the region $\frac{2(d-j+1)}{d-j}<q \leq \frac{2(d-j)}{d-j-1}$, for $0 \leq j \leq d-2$.  
\end{lemma}

This completes the proof of Theorem~\ref{T:pos} in the range $q\leq\frac{10}3$.  

\begin{proof}
By H\"older's inequality, for $q>p$,
$$
\|\scriptE^\ell_g\|_{L^p \to L^q} \leq (l_1\cdots l_d)^{\frac1q-\frac1p}\|\scriptE^\ell_g\|_{L^q \to L^q},
$$
so it suffices to verify the theorem on the line $q=p$.  By the Drury--Guo dimension reduction argument used in the proof of Lemma~\ref{L:DG}, 
$$
\|\scriptE^\ell_g\|_{L^q \to L^q} \leq (l_1\cdots l_j)^{\frac1{q'}-\frac1q}\sup_{\xi' \in Q^{\ell'}}\|\scriptE^{\ell''}_{g_{\xi'}}\|_{L^q \to L^q}.
$$
This reduces matters to the case $j=0$.  By parabolic rescaling (which we recall leaves Conjecture~\ref{C:rect conj} invariant), it suffices to consider the case when $l_1=1$.  

Write $q=\tfrac{2(d-\theta+1)}{d-\theta}$, with $0 < \theta \leq 1$.  By hypothesis and Lemma~\ref{L:interstices}, for all $0 < \nu < \theta$,
$$
\|\scriptE^\ell_g\|_{L^{p_\nu} \to L^q} \lesssim 1, \qtq{where} p_\nu:=(\tfrac{d-\nu}{d-\nu+2} q)'.
$$
Thus by H\"older's inequality,
$$
\|\scriptE^\ell_g\|_{L^q \to L^q} \lesssim (l_1 \cdots l_d)^{\frac1{p_\nu}-\frac1q} \lesssim l_d^{(d-1)(\frac1{p_\nu}-\frac1q)}.
$$
Setting $\eps:= (d-1)(\tfrac1{p_\nu}-\tfrac1q)$ and sending $\nu\nearrow \theta$ completes the proof.  
\end{proof}

We now turn to the fully unconditional results.  

\begin{lemma} \label{L:q>4}
Theorem~\ref{T:pos} holds in the region $q > 4$.  
\end{lemma}

\begin{proof}
The proof is a direct deduction from Lemma~\ref{L:DG} (applied in the case $q=3p'$) via H\"older's inequality:  
$$
\|\scriptE^\ell_g f\|_q \lesssim (l_1\cdots l_{d-1})^{\frac3{q}-\frac1q}\|f\|_{(\frac q3)'} \lesssim (l_1\cdots l_{d-1})^{\frac3{q}-\frac1q}|Q^\ell|^{1-\frac 3q-\frac1p}\|f\|_p,
$$
and the right hand side equals that of \eqref{E:q > 4}.  
\end{proof}

\begin{lemma}
Theorem~\ref{T:pos} holds in the region $q>\frac{10}3$.  
\end{lemma}

\begin{proof}
Let $q>\tfrac{10}3$.  By Lemma~\ref{L:q>4}, we may assume that $q \leq 4$.  By Lemma~\ref{L:DG}, we may assume that $q>2p'$.  By adapting the first part of the proof of Lemma~\ref{L:q leq p} (which reduces the dimension and reduces to the case $q<p$), we may assume that $d=2$ and $q>p$.  By parabolic rescaling, we may assume that $1=l_1 \leq l_2$.  In summary, it remains to prove that when $d=2$, $1=l_1 \leq l_2$, $4 \leq q < \tfrac{10}3$, and $q>\max\{2p',p\}$, we have
$$
\|\scriptE^\ell_g\|_{L^p \to L^q} \lesssim 1.  
$$
The above conditions on $p,q$ imply that $p>2$.  By interpolation with the (known) inequality $\|\scriptE^\ell_g\|_{L^p \to L^q} \lesssim 1$ on the line $q=3p'$, $q>p$, we may assume, in addition, that $\tfrac8q+\tfrac2p > 3$.  Finally, by real interpolation, it suffices to bound $\|\scriptE^\ell_g f_\Omega\|_q$ for $|f_\Omega| \sim \chi_\Omega$, $\Omega$ a finite measure set.

We adapt the argument of Tao--Vargas--Vega \cite{TVV}.  Making a Whitney decomposition,
$$
Q^\ell \times Q^\ell = \bigcup_{N =0}^\infty \bigcup_{\tau \sim \tau' \in \scriptD_N} \tau \times \tau',
$$
where $\scriptD_N$ denotes a finitely overlapping collection of width 1, height $2^N$ rectangles contained in $Q^\ell$ and $\tau \sim \tau'$ if $N=0$ and $\dist(\tau,\tau') \lesssim 1$ or $N>0$ and $\dist(\tau,\tau') \sim 2^N$.  Making a partition of unity, we can write
$$
(\scriptE f_\Omega)^2 = \sum_{N=0}^\infty \sum_{\tau \sim \tau' \in \scriptD_N} \scriptE f_{\Omega \cap \tau} \scriptE f_{\Omega \cap \tau'},
$$
where $|f_{\Omega \cap \tau}|$ is bounded by the characteristic function of $\Omega \cap \tau$.  

Letting $\widetilde\tau := \{(g(\xi),\xi):\xi \in \tau\}$, we see that the convex hulls of the $\widetilde\tau+\widetilde{\tau}'$, with $\tau \sim \tau' \in \scriptD_N$, are finitely overlapping as $\tau,\tau',N$ vary.  Thus by Lemma~6.1 of \cite{TVV},
$$
\|\scriptE f_\Omega\|_q^q \lesssim \sum_{N=0}^\infty \sum_{\tau \sim \tau' \in \scriptD_N} \|\scriptE f_{\Omega \cap\tau}\scriptE f_{\Omega \cap \tau'}\|_{\frac q2}^{\frac q2}.
$$

A volume preserving affine transformation maps the $\{(g(\xi),\xi):\xi \in \tau\}$, $\tau \in \scriptD_0$ to surfaces elliptic over $Q^\1$.  Thus in the case $N=0$, we may apply Cauchy--Schwarz, the  inequality $\|\scriptE^{\1}_g\|_{L^p \to L^q}\lesssim 1$ (which follows immediately from Theorem~\ref{T:elliptic thm} via H\"older's inequality), and H\"older's inequality to see that
\begin{align*}
\sum_{\tau \sim \tau' \in \scriptD_0} \|\scriptE f_{\Omega \cap \tau}\scriptE f_{\Omega \cap \tau'}\|_{\frac q2}^{\frac q2} 
\lesssim  \sum_{\tau \in \scriptD_0} |\Omega \cap \tau|^{\frac qp} 
\lesssim |\Omega|^{\frac qp}.
\end{align*}
It thus remains to bound the terms with $N>0$.  

When $N>0$ and $\tau \sim \tau'$, rescaling Lemma~\ref{L:Candy} implies that 
$$
\|\scriptE f_\tau \scriptE f_{\tau'}\|_{\frac q2} \lesssim 2^{N(2-\frac 8q)}\|f_\tau\|_2\|f_{\tau'}\|_2.
$$
Thus we obtain by following the  argument of \cite{TVV} that
\begin{equation} \label{E:N>1}
\begin{aligned}
&\sum_{N=1}^\infty \sum_{\tau \sim \tau' \in \scriptD_N} \|\scriptE f_{\Omega \cap \tau}\scriptE  f_{\Omega \cap \tau'}\|_{\frac q2}^{\frac q2} 
\lesssim \sum_{N=1}^\infty 2^{N(q-4)} \sum_{\tau \sim \tau' \in \scriptD_N}|\Omega \cap \tau|^{\frac q4}|\Omega \cap \tau'|^{\frac q4}\\
&\qquad  \lesssim \sum_{N=1}^\infty 2^{N(q-4)} \sum_{\tau \in \scriptD_N} |\Omega \cap \tau|^{\frac q2} \lesssim \sum_{N=1}^\infty 2^{N(q-4)}\min\{|\Omega|,2^N\}^{\frac q2-1}|\Omega|;
\end{aligned}
\end{equation} 
here we have used the fact that $|\Omega \cap \tau| \leq \min\{|\Omega|,|\tau|\} = \min\{|\Omega|,2^N\}$, for each $\tau \in \scriptD_N$.  

Our proof now bifurcates into two cases, $|\Omega|\geq 1$ and $|\Omega| \leq 1$.  If $|\Omega| \leq 1$, the right hand side of \eqref{E:N>1} is bounded by
$$
|\Omega|^{\frac q2} \leq |\Omega|^{\frac qp},
$$
since $p \geq 2$.  If $|\Omega| \geq 1$, the right hand side of \eqref{E:N>1} is bounded by
$$
\sum_{N=1}^{\log|\Omega|} 2^{N(\frac{3q}2-5)}|\Omega| + \sum_{N=\log|\Omega|}^\infty 2^{N(q-4)}|\Omega|^{\frac q2} \lesssim |\Omega|^{\frac{3q}2-4} \leq |\Omega|^{\frac qp},
$$
where we have used $\frac 2p+\frac 8q \geq 3$ in the last inequality.  
\end{proof}

\section{The application:  The proof of Theorem~\ref{T:E_beta}}

We turn now to the proof of Theorem~\ref{T:E_beta}, to which we devote the entirety of this section.  By the triangle inequality and the symmetry of changing the sign of any $\xi_i$, it suffices to bound the operator
$$
\scriptE_\beta f(t,x):= \int_{(0,1]^d} e^{i(t,x)(|\xi_1|^{\beta_1}+\cdots+|\xi_d|^{\beta_d},\xi)}f(\xi)\,d\xi.
$$

We begin by making a dyadic decomposition,
$$
(0,1]^d = \bigcup_{k \in \N^d}R^k, \qquad R^k := \{\xi:\xi_i \sim 2^{-2k_i}, \: i=1,\ldots,d\},
$$
which induces the decomposition
$$
\scriptE_\beta = \sum_{k \in \N^d} \scriptE_\beta^k, \qquad \scriptE_\beta^k f(t,x):= \int_{R^k} e^{i(t,x)(|\xi_1|^{\beta_1}+\cdots+|\xi_d|^{\beta_d},\xi)}f(\xi)\,d\xi.
$$
For $\sigma$ a permutation in $S_d$, we define
$$
\scriptK_\beta^\sigma:=\{k \in \N^d : k_{\sigma(1)}\beta_{\sigma(1)} \geq \cdots \geq k_{\sigma(d)}\beta_{\sigma(d)}\}
$$
and $\scriptE^\sigma_\beta := \sum_{k \in \scriptK_\beta^\sigma} \scriptE_\beta^k$.  By the triangle inequality, it suffices to bound each $\scriptE^\sigma_\beta$.  We will do this by first bounding the partial sums
\begin{gather*}
\scriptE_\beta^{\sigma,k_\sigma(1)} := \sum_{k' \in \scriptK_\beta^\sigma(k_{\sigma(1)})} \scriptE_\beta^k,\\
\scriptK_\beta^\sigma(k_{\sigma(1)}) := \{(k_{\sigma(j)}')_{j=2}^d \in \N^{d-1} : k_{\sigma(1)}\beta_{\sigma(1)} \geq k_{\sigma(2)}'\beta_{\sigma(2)} \geq \cdots \geq k_{\sigma(d)}'\beta_{\sigma(d)}\}.
\end{gather*}

We restate Conjecture~\ref{C:rect conj} and Proposition~\ref{P:sharp E_beta}, as they apply to the $\scriptE_\beta^k$.  

\begin{lemma}\label{L:Ek upper}
Let $p,q \in [1,\infty]$, and assume that the conclusions of Conjecture~\ref{C:rect conj} hold for this exponent pair.  Let $k \in \scriptK_\beta^\sigma$. If $q > p$ and $q=\tfrac{d-j-\theta+2}{d-j-\theta} p'$, for some $0 \leq j < d$ and $0 \leq \theta \leq 1$, then
\begin{align} \label{E:Ek upper:q > p}
&\|\scriptE_\beta^k\|_{L^p \to L^q} \lesssim \\\notag
&\left[\bigl(\prod_{m=1}^j 2^{-2k_{\sigma(m)}}\bigr) \bigl(2^{k_{\sigma(j+1)}\beta_{\sigma(j+1)}[(1-\theta)-\frac2{\beta_{\sigma(j+1)}}]}\bigr)\bigl(\prod_{m=j+2}^d 2^{k_{\sigma(m)}\beta_{\sigma(m)}(1-\frac2{\beta_{\sigma(m)}})}\bigr)\right]^{\frac1{p'}-\frac1q}.
\end{align}
Additionally,  
\begin{align}\notag
\|\scriptE_{\beta}^k\|_{L^p \to L^q} \lesssim_{\eps} & \bigl(\prod_{m=1}^j 2^{-2k_{\sigma(m)}(\frac1{p'}-\frac1q)}\bigr)\bigl(2^{k_{\sigma(j+1)}\beta_{\sigma(j+1)}[(1-\theta)(1-\frac2q)-\frac2{\beta_{\sigma(j+1)}}(\frac1{p'}-\frac1q)+\eps]}\bigr)\\ \label{E:Ek upper:q leq p}
& \times \bigl(\prod_{m=j+2}^d 2^{k_{\sigma(m)}\beta_{\sigma(m)}[(1-\frac2q)-\frac2{\beta_{\sigma(m)}}(\frac1{p'}-\frac1q)]}\bigr)2^{-k_{\sigma(d)}\beta_{\sigma(d)}\eps}.
%
\end{align}
for $q = \tfrac{2(d-j-\theta+1)}{d-j-\theta} \leq p$, $0 < \theta \leq 1$, $\eps > 0$, and $j=0,\ldots,d-1$. 
\end{lemma}

The result is true without the loss $2^{\epsilon (k_{\sigma(j+1)}\beta_{\sigma(j+1)}-k_{\sigma(d)}\beta_{\sigma(d)})}$ in the range $p \geq q >4$, but, since this loss is harmless for our application, we have left it in to simplify the statement.  

\begin{proof}
The lemma is proved by introducing coordinates, $\eta_i = 2^{(2-\beta_{\sigma(i)})k_{\sigma(i) }}\xi_{\sigma(i)}$, $\xi \in R^k$, producing the function
$$
g_\beta^k(\eta) = \sum_{i=1}^d 2^{k_{\sigma(i)}(\beta_{\sigma(i)}-2)\beta_{\sigma(i)}}|\eta_i|^{\beta_{\sigma(i)}},
$$
which is elliptic (to arbitrary order) over a rectangle $\{\eta_i \sim 2^{-k_{\sigma(i)}\beta_{\sigma(i)}}\}$.  In the notation of Conjecture~\ref{C:rect conj}, this rectangle is congruent to $Q^\ell$, where 
$$
\ell := \bigl(2^{-k_{\sigma(1)}\beta_{\sigma(1)}},\ldots,2^{-k_{\sigma(d)}\beta_{\sigma(d)}}\bigr).
$$
\end{proof}

\begin{proof}[Proof of Theorem~\ref{T:E_beta}]
We will give the details of the proof only in the case $\beta_i \neq 2$ for all $i$.  In the case that $\beta_i=2$ for some $i$, we may use a Galilean transformation in those coordinates $\xi_i$ with $\beta_i = 2$ to see that 
\begin{equation} \label{E:simplified betai=2}
\|\scriptE_\beta\|_{L^p \to L^q} \lesssim \|{\sum_{k \in \N^d}}' \scriptE_\beta^k\|_{L^p \to L^q},
\end{equation}
where the prime indicates a sum taken over those $k \in \N^d$ with $k_i = 1$, for all $i$ such that $\beta_i=2$.  Since the difficulty in the general case lies in summing over those $k_i$ such that $\beta_i \neq 2$, we will give the complete details only in the case that $\beta_i \neq 2$ for all $i$.  The change needed to handle the general case is just notational.  

For most cases, we will use an interpolation lemma whose hypotheses will necessitate boundedness of $\scriptE_\beta^k$ as an operator from $L^{\tilde p}$ to $L^{\tilde q}$ for $(\tilde p,\tilde q)$ lying in a neighborhood (in $\R^2$) of $(p,q)$.  Thus we begin by dispensing with those cases wherein $(p,q)$ lies on the boundary of the region $T_d$ defined in \eqref{E:def Td}.  

The case $q=\infty$ is elementary: $\scriptE_\beta : L^p \to L^\infty$ for all $p \geq 1$, by H\"older's inequality.  The case $p=\infty>q$ may arise under condition (ii) or (iv), but the claimed bounds for such pairs follow from the claimed bounds with finite $p$ by H\"older's inequality, except possibly for the point $(p,q)=(\infty,1+\frac 1{J_d})$, which can arise under condition (iv).  The case $p<q=\tfrac{d+2}d p'$ is a little more involved.  On the one hand, in the notation of Lemma~\ref{L:Ek upper}, $j+\theta = 0$, so \eqref{E:Ek upper:q > p} reads 
\begin{equation} \label{E:small beta pos}
\|\scriptE_\beta^k\|_{L^p \to L^{\frac{d+2}d p'}} \lesssim \prod_{j=1}^d 2^{-k_{\sigma(i)}(2-\beta_{\sigma(i)})(\frac1{p'}-\frac1q)}.
\end{equation}
On the other hand, from the hypothesis $\beta_i \neq 2$ and the ordering $\beta_1 \geq \cdots \geq \beta_d$, the difference $(J_n+\tfrac{d-n}2) -(J_{n-1}+\tfrac{d-n+1}2) = \tfrac1{\beta_n}-\tfrac12$ is increasing in $n$ and never zero.  Therefore, the case $\tfrac q{p'} = \tfrac{d+2}d = 1+\frac1{J_0+\frac{d-0}2}$ of Condition (i) of the theorem is only possible when $n \mapsto J_n + \tfrac{d-n}2$ has a strict minimum at 0, i.e.\ when $\beta_i<2$ for all $i$.  In this case, all of the exponents in \eqref{E:small beta pos} are negative, and it is elementary to sum.  

We will specifically address the case $(p,q) = (\infty,1+\tfrac1{J_d})$ at the end of the proof, but for now, we may assume that the bounds in Lemma~\ref{L:Ek upper} hold for exponent pairs in a neighborhood (in $\R^2$) of $(p,q)$.  By the reductions above and real interpolation, it suffices to prove that the bounds expressed in the theorem hold for all pairs $q>\max\{\tfrac{d+2}d p', \tfrac{2(d+1)}d\}$ obeying, in addition, one of the conditions (i), (ii and $p<\infty$), or (iv).  

We now complete the argument in the case of (i).  Let $p < q = \tfrac{d-j-\theta+2}{d-j-\theta} p'$, for some $0 \leq j < d$ and $0 \leq \theta < 1$.  We start by proving bounds for the $\scriptE_\beta^{\sigma,k_\sigma(1)}$ for exponent pairs $(p,q)$ lying in a neighborhood of some pair obeying (i); thus we do not yet assume that (i) holds.  

Since $1-\tfrac2{\beta_i} \neq 0$ for all $i$, applying \eqref{E:Ek upper:q > p} and summing a geometric series,
\begin{align*}
&\|\scriptE_\beta^{\sigma,k_{\sigma(1)}}\|_{L^p \to L^q} \leq \sum_{k' \in \scriptK_\beta^\sigma(k_{\sigma(1)})} \|\scriptE_\beta^k\|_{L^p \to L^q}\\
&\qquad \lesssim {\sum_{k_{\sigma(2)}}}' \cdots {\sum_{k_{\sigma(j+1)}}}'\left[\bigl(\prod_{m=1}^j 2^{-2k_{\sigma(m)}}\bigr)2^{k_{\sigma(j+1)}\beta_{\sigma(j+1)}[-\theta+\sum_{i=j+1}^d(1-\frac2{\beta_{\sigma(i)}})_+]}\right]^{\frac1{p'}-\frac1q},
\end{align*}
where the $'$'s indicate sums over $1 \leq k_{\sigma(m+1)} \leq \tfrac{k_{\sigma(m)}\beta_{\sigma(m)}}{\beta_{\sigma(m+1)}}$.  For $0 \leq l < j$, if $-\theta-l+\sum_{i=j+1-l}^d(1-\tfrac2{\beta_{\sigma(i)}})_+ > 0$,
\begin{align*}
&{\sum_{k_{\sigma(j+1-l)}}}'\left[\bigl(\prod_{m=1}^{j-l} 2^{-2k_{\sigma(m)}}\bigr)2^{k_{\sigma(j+1-l)}\beta_{\sigma(j+1-l)}[-\theta-l+\sum_{i=j+1-l}^d(1-\frac2{\beta_{\sigma(i)}})_+]}\right]^{\frac1{p'}-\frac1q}\\
&\qquad \lesssim 
\left[ \bigl(\prod_{m=1}^{j-l-1} 2^{-2k_{\sigma(m)}}\bigr)2^{k_{\sigma(j-l)}\beta_{\sigma(j-l)}[-\theta-l-1+\sum_{i=j-l}^d(1-\frac2{\beta_{\sigma(i)}})_+]}\right]^{\frac1{p'}-\frac1q}.
\end{align*}
By the preceding and a simple induction argument,
\begin{equation} \label{E:E beta sigma k I}
\|\scriptE_\beta^{\sigma,k_{\sigma(1)}}\|_{L^p \to L^q} \lesssim 2^{k_{\sigma(1)}\beta_{\sigma(1)}[-\theta-j+\sum_{i=1}^d(1-\frac2{\beta_i})_+](\frac1{p'}-\frac1q)},
\end{equation}
if $-\theta-l+\sum_{i=j+1-l}^d(1-\tfrac2{\beta_{\sigma(i)}})_+>0$, for all $0 \leq l < j$, while if $-\theta-l+\sum_{i=j+1-l}^d(1-\tfrac2{\beta_{\sigma(i)}})_+ \leq 0$ for some $0 \leq l < j$, which we assume to be the least such $l$,
\begin{equation} \label{E:E beta sigma k II}
\|\scriptE_\beta^{\sigma,k_{\sigma(1)}}\|_{L^p \to L^q} \lesssim {\sum_{k_{\sigma(2)}}}' \cdots {\sum_{k_{\sigma(j-l)}}}' k_{\sigma(j-l)} \prod_{m=1}^{j-l} 2^{-2k_{\sigma(m)}(\frac1{p'}-\frac1q)}.
\end{equation}

Suppose now that condition (i) holds.  We may sum the right side of \eqref{E:E beta sigma k I} in $k_{\sigma(1)}$ whenever $\theta+j>\sum_{i=1}^d(1-\tfrac2{\beta_i})_+$, and the right side of \eqref{E:E beta sigma k II} may be summed in $k_{\sigma(1)}$ unconditionally.  Condition (i) for our $p,q$ is, after a bit of arithmetic, equivalent to $\theta+j \geq \sum_{i=1}^d(1-\tfrac2{\beta_i})_+$.  Thus it remains to consider the case
\begin{equation} \label{E:theta + l ineq}
\theta+l<\sum_{i=j+1-l}^d(1-\tfrac2{\beta_{\sigma(i)}})_+, \: 0 \leq l < j, \qtq{and} \theta+j=\sum_{i=1}^d(1-\tfrac2{\beta_i})_+.
\end{equation}

Let $j_0 \leq j = j_1$ and $0 < \theta_0,\theta_1 < 1$ with $j_0+\theta_0 < j+\theta < j_1+\theta_1$ and $|j+\theta-(j_i+\theta_i)|$ sufficiently small, $i=0,1$.  Then with $p_i':= \tfrac{d-j_i-\theta_i}{d-j_i-\theta_i+2}q$, $q>p_i$, and inequality \eqref{E:Ek upper:q > p} holds at $(p_i,q)$.  Furthermore, since $j_1=j$, \eqref{E:theta + l ineq} implies
$$
\theta_1+l < \sum_{i=j_1+1-l}^d(1-\tfrac2{\beta_{\sigma(i)}})_+, \qquad 0 \leq l < j_1,
$$
provided $|\theta-\theta_1|$ is sufficiently small, and we may argue similarly for $\theta_0$ when $j_0=j$.  If $j_0 < j$, we may assume that $j_0=j-1$, so
$$
\theta_0+l < \theta+l+1 < \sum_{i=j+1-(l+1)}^d(1-\tfrac2{\beta_{\sigma(i)}})_+ = \sum_{i=j_0+1-l}^d(1-\tfrac2{\beta_{\sigma(i)}})_+, \qquad 0 \leq l+1 < j = j_0+1.  
$$
Therefore by \eqref{E:E beta sigma k I}, 
$$
\|\scriptE_\beta^{\sigma,k_{\sigma(1)}}\|_{L^{p_i} \to L^{q}} \lesssim 2^{\alpha_i k_{\sigma(1)}},
$$
where
$$
\alpha_i:=\beta_{\sigma(1)}(-\theta_i-j_i+\sum_{i=1}^d(1-\tfrac2{\beta_i})_+)(\tfrac1{p_i'}-\tfrac1q), \qquad i=0,1.
$$
We observe that $\alpha_0 > 0 > \alpha_1$.  Thus for $ f_\Omega$ comparable to the characteristic function of a measurable set $\Omega$,
\begin{align*}
\|\sum_{k_{\sigma(1)}} \scriptE_\beta^{\sigma,k_{\sigma(1)}} f_\Omega\|_{L^q} 
&\lesssim \sum_{k_{\sigma(1)}} \min\{2^{\alpha_0 k_{\sigma(1)}}|\Omega|^{\frac1{p_0}},2^{\alpha_1 k_{\sigma(1)}}|\Omega|^{\frac1{p_1}}\}\\
&\lesssim |\Omega|^{\frac{\frac{\alpha_0}{p_1}-\frac{\alpha_1}{p_0}}{\alpha_0-\alpha_1}} = |\Omega|^{\frac1p}
\end{align*}
(some arithmetic is needed for the last equality).  This implies the restricted weak type inequality, and so completes the proof of the restricted-weak type inequality in the case $q > p$. Since the Riesz diagram lacks any vertex in the region $q > p$, by real interpolation, the proof is complete for $q > p$.

We now turn to the case $q = \tfrac{2(d-j-\theta+1)}{d-j-\theta} \leq p$.  For any integer $N$, 
$$
\max_{0 \leq n \leq N} n(1-\tfrac2q)-2J_n(\tfrac1{p'}-\tfrac1q) = \sum_{i=1}^N[(1-\tfrac2q)-\tfrac2{\beta_i}(\tfrac{1}{p'}-\tfrac1q)]_+.
$$
Thus (iv) can be rewritten as 
\begin{gather*}
q \leq p \leq \infty,\, \qtq{and}\, \sum_{i=1}^{d-1}[(1-\tfrac2q)-\tfrac2{\beta_i}(\tfrac1{p'}-\tfrac1q)]_+ < d(1-\tfrac2q)-\tfrac2q, \\\qtq{and}\, \sum_{i=1}^d[(1-\tfrac2q)-\tfrac2{\beta_i}(\tfrac1{p'}-\tfrac1q)] = d(1-\tfrac2q)-\tfrac2q.
\end{gather*}
This implies that $(1-\tfrac2q)-\tfrac2{\beta_d}(\tfrac1{p'}-\tfrac1q) > 0$, which, by our assumption that $-\tfrac1{\beta_1} \geq \cdots \geq -\tfrac1{\beta_d}$, further implies (since $p' < q$) that $(1-\tfrac2q)-\tfrac2{\beta_i}(\tfrac1{p'}-\tfrac1q) > 0$, for all $i$.  Collecting these observations, and making similar (but simpler) manipulations in the case of (ii), we may rewrite conditions (ii and $p<\infty$) and (iv) as
$$
\begin{cases}
{\text{(ii')}}\:q\leq p<\infty, \qtq{and} \sum_{i=1}^d[(1-\tfrac2q)-\tfrac2{\beta_i}(\tfrac1{p'}-\tfrac1q)]_+ < d(1-\tfrac2q)-\tfrac2q,\\
{\text{(iv')}}\:q\leq p\leq \infty, \qtq{and} (1-\tfrac2q)-\tfrac2{\beta_i}(\tfrac1{p'}-\tfrac1q) > 0, \qtq{for all $i$, and} \tfrac{1+J_d}q = \tfrac{J_d}{p'}.
\end{cases}
$$

Let us now assume that we are in case (ii').  Then by \eqref{E:Ek upper:q leq p},
\begin{equation} \label{E:interior ii'}
\begin{aligned}
\|\scriptE_\beta^{\sigma,k_{\sigma(1)}}\|_{L^p \to L^q}& \lesssim_\eps 2^{\eps k_{\sigma(1)}\beta_{\sigma(1)}}{\sum_{k_{\sigma(2)}}}' \cdots {\sum_{k_{\sigma(j+1)}}}' \bigl[\bigl(\prod_{m=1}^j 2^{-2k_{\sigma(m)}(\frac1{p'}-\frac1q)}\bigr)\\
& \times 2^{k_{\sigma(j+1)}\beta_{\sigma(j+1)}(-\theta(1-\frac2q)+\sum_{m=j+1}^d[(1-\frac2q)-\frac2{\beta_{\sigma(m)}}(\frac1{p'}-\frac1q)]_+)}\bigr],
\end{aligned}
\end{equation}
where the extra factor in front accounts for the loss in $2^{k_{\sigma(j+1)}\beta_{\sigma(j+1)}}$ coming from \eqref{E:Ek upper:q leq p}.  Mimicking the inductive argument from before, if 
$$
-(\theta+l)(1-\tfrac2q) + \sum_{m=j+1-l}^d [(1-\tfrac2q)-\tfrac2{\beta_{\sigma(m)}}(\tfrac1{p'}-\tfrac1q)]_+ \leq 0,
$$
for any $0 \leq l < j$, the right hand side of \eqref{E:interior ii'} is summable in $k_{\sigma(1)}$.  Also as before, if 
$$
-(\theta+l)(1-\tfrac2q) + \sum_{m=j+1-l}^d [(1-\tfrac2q)-\tfrac2{\beta_{\sigma(m)}}(\tfrac1{p'}-\tfrac1q)]_+ > 0,\qquad 0 \leq l < j,
$$
then we have
\begin{equation} \label{E:interior ii''}
\|\scriptE_\beta^{\sigma,k_{\sigma(1)}}\|_{L^p \to L^q} \lesssim_\eps 2^{k_{\sigma(1)}\beta_{\sigma(1)}(-(\theta+j)(1-\frac2q)+\sum_{m=1}^d[(1-\frac2q)-\frac2{\beta_m}(\frac1{p'}-\frac1q)]_++\eps)}.
\end{equation}
By inserting the definition of $q$ into (ii') and taking $\eps$ sufficiently small, the exponent on the right hand side of \eqref{E:interior ii''} is negative, and so we may again sum in $k_{\sigma(1)}$.  

We will turn in a moment to (iv'), but for now we assume only that $q = \tfrac{2(d-j-\theta+1)}{d-j-\theta} \leq p$ and 
$$
(1-\tfrac2q)-\tfrac2{\beta_i}(\tfrac1{p'}-\tfrac1q) > 0, \qtq{for all} i.  
$$
Taking $\eps < (1-\tfrac2q)-\tfrac2{\beta_i}(\tfrac1{p'}-\tfrac1q)$, $i=1,\ldots,d$, \eqref{E:Ek upper:q leq p} implies
\begin{equation} \label{E:E beta sigma q leq p I}
\begin{aligned}
\|\scriptE_\beta^{\sigma,k_{\sigma(1)}}\|_{L^p \to L^q} \lesssim &{\sum_{k_{\sigma(2)}}}' \cdots {\sum_{k_{\sigma(j+1)}}}' \bigl(\prod_{m=1}^j 2^{-2k_{\sigma(m)}(\frac1{p'}-\frac1q)}\bigr)\\
&\times 2^{k_{\sigma(j+1)}\beta_{\sigma(j+1)}(-\theta(1-\frac2q)+\sum_{m=j+1}^d[(1-\frac2q)-\frac2{\beta_{\sigma(m)}}(\frac1{p'}-\frac1q)])}.  
\end{aligned}
\end{equation}
Again arguing as before, if
\begin{equation} \label{E:iib condition}
-(\theta+l)(1-\tfrac2q)+\sum_{m=j+1-l}^d[(1-\tfrac2q)-\tfrac2{\beta_{\sigma(m)}}(\tfrac1{p'}-\tfrac1q)] > 0, \qquad 0 \leq l < j,
\end{equation}
then \eqref{E:E beta sigma q leq p I} implies
\begin{equation}\label{E:E beta sigma q leq p II}
\|\scriptE_\beta^{\sigma,k_{\sigma(1)}}\|_{L^p \to L^q} \lesssim 2^{k_{\sigma(1)}\beta_{\sigma(1)}(-(\theta+j)(1-\frac2q)+\sum_{m=1}^d[(1-\frac2q)-\frac2{\beta_m}(\frac1{p'}-\frac1q)])},
\end{equation}
while if \eqref{E:iib condition} fails for some $0 \leq l < j$, the right hand side of \eqref{E:E beta sigma q leq p I} may be summed in $k_{\sigma(1)}$.  

Now we assume that (iv') holds.  The equation $\frac{1+J_d}{q}=\frac{J_d}{p'}$ can be rewritten, after a little algebra, as 
$$
-(\theta+j)(1-\tfrac2q)+\sum_{m=1}^d[(1-\tfrac2q)-\tfrac2{\beta_m}(\tfrac1{p'}-\tfrac1q)] = 0.
$$  
Our analysis now breaks into three cases.  If $q < p<\infty$, we choose $q < p_0 < p < p_1$ with $|p-p_i|$ sufficiently small that \eqref{E:iib condition} holds with $(p_i,q)$ in place of $(p,q)$.  For $ f_\Omega$ comparable to a characteristic function,
\begin{align*}
\|\scriptE_\beta^\sigma  f_\Omega\|_{L^q} \lesssim \sum_{k=1}^\infty \min_{i=0,1}\bigl\{2^{k\beta_{\sigma(1)}(-(\theta+j)(1-\frac2q)+\sum_{m=1}^d[(1-\frac2q)-\frac2{\beta_m}(\frac1{p_i'}-\frac1q)])}|\Omega|^{\frac1{p_i}}\bigr\} \lesssim |\Omega|^{\frac1p},
\end{align*}
which implies the restricted weak type inequality that we want.

If $q=p$, Condition (iii) holds, so we must prove a strong type inequality.  On the line $\frac{1+J_d}{q}=\frac{J_d}{p'}$, equation \eqref{E:iib condition} continues to hold for some $q<p$, and condition (i) holds for $q>p$. The strong type inequality follows by real interpolation and estimates already proved.

Finally, if $(p,q)=(\infty,\frac 1{J_d}+1)$, taking $j_0+\theta_0 < j+\theta < j_1+\theta_1$ and $q_i := \tfrac{2(d-j_i-\theta_i+1)}{d-j_i-\theta_i}$, for $|j_i+\theta_i-(j+\theta)|$ sufficiently small \eqref{E:iib condition} holds.  We may rewrite \eqref{E:E beta sigma q leq p II} as
$$
\|\scriptE_\beta^{\sigma,k_{\sigma(1)}}\|_{L^\infty \to L^{q_i}} \lesssim 2^{k_{\sigma(1)}\beta_{\sigma(1)}[(d-\theta_i-j_i-2J_d)(1-\frac2{q_i})-\frac{2J_d}{q_i}]}.
$$
As the exponent on the right is positive for $i=0$ and negative for $i=1$, we see (after some arithmetic) that for $ f_{\Omega_1},  f_{\Omega_2}$ comparable to characteristic functions,
$$
\langle \scriptE_\beta^\sigma  f_{\Omega_1}, f_{\Omega_2}\rangle \lesssim \sum_k \min\{2^{k\beta_{\sigma(1)}[(d-\theta_i-j_i-2J_d)(1-\frac2{q_i})-\frac{2J_d}{q_i}]}|\Omega_2|^{\frac1{q_i'}}\} \lesssim |\Omega_2|^{\frac1{q'}},
$$
which implies the claimed restricted weak type inequality in the remaining case.  
\end{proof}


\section{The negative result:  Proof of Proposition~\ref{P:sharp E_beta}}

We use the notation established at the beginning of the previous section.  Rescaling the lower bounds in Theorem~\ref{T:neg} (analogously to the proof of Lemma~\ref{L:Ek upper}) yields the following lower bounds on the $\scriptE_\beta^k$.  

\begin{lemma}\label{L:Ek lower}
Assume that $k_1\beta_1\geq k_2\beta_2\geq\cdots\geq k_d\beta_d$. If $q=\tfrac{d-j-\theta+2}{d-j-\theta} p'>p$, for some $0 \leq j < d$ and $0 \leq \theta \leq 1$, then
\begin{align} \label{E:EK lower:q > p}
&\|\scriptE_{\beta}^k\|_{L^p \to L^q}^{\rm{RWT}} \gtrsim  
\left[\bigl(\prod_{i=1}^j 2^{-2k_{i}}\bigr) \bigl(2^{k_{j+1}\beta_{j+1}[(1-\theta)-\frac2{\beta_{j+1}}]}\bigr)\bigl(\prod_{i=j+2}^d 2^{k_{i}\beta_{i}(1-\frac2{\beta_{i}})}\bigr)\right]^{\frac1{p'}-\frac1q}.
\end{align}
Additionally, if $q = \tfrac{2(d-j-\theta+1)}{d-j-\theta} \leq p$, for some $0 < \theta \leq 1$ and $j=0,\ldots,d-1$, 
\begin{align}\notag
\|\scriptE_{\beta}^k\|_{L^p \to L^q}^{\rm{RWT}} \gtrsim & \bigl(\prod_{i=1}^j 2^{-2k_{i}(\frac1{p'}-\frac1q)}\bigr)\bigl(2^{k_{j+1}\beta_{j+1}[(1-\theta)(1-\frac2q)-\frac2{\beta_{j+1}}(\frac1{p'}-\frac1q)]}\bigr)\\ \label{E:EK lower:q leq p d>1}
& \times \bigl(\prod_{i=j+2}^d 2^{k_{i}\beta_{i}[(1-\frac2q)-\frac2{\beta_{i}}(\frac1{p'}-\frac1q)]}\bigr)\tilde \alpha\bigl(k_{j+1}\beta_{j+1}-k_{d}\beta_{d}\bigr),
\end{align}
for some increasing $\tilde \alpha$ depending on $p,q,d$, satisfying $\tilde \alpha(0)=1$ and $\alpha(r)\rightarrow \infty$ as $r\rightarrow \infty$. 
\end{lemma}

\begin{proof}[Proof of Proposition~\ref{P:sharp E_beta}]
Let $(p,q) \in [1,\infty]^2$, and assume that none of the conditions (i-iv) hold.  We may assume that $(p,q) \in T_d$ and $p \neq 1$, $q \neq \infty$.  We may define $j,\theta$, depending on $(p,q)$, such that $q$ can be written in one of the forms given in Lemma~\ref{L:Ek lower}. 

Failure of conditions (i-iv) for $(p,q) \in T_d$ leads to a choice of an integer $n \geq 1$.  Namely, if $q>p$, we choose $1 \leq n\leq d$ such that $\tfrac{q}{p'} < 1+\tfrac1{J_n+\frac{d-n}2}$.  If $q \leq p$, we choose $n=d$ if $\tfrac{1+J_d}q > \tfrac{J_d}{p'}$, and otherwise choose $n<d$ such that $\tfrac{1+J_n+d-n}q \geq \tfrac{J_n}{p'}+\tfrac{d-n}2$.   A bit of arithmetic shows that in any of these cases, $n \geq j+\theta$.  

Let $N>\beta_1$ sufficiently large and define $\tilde k = (\lfloor \tfrac N{\beta_1}\rfloor,\ldots,\lfloor \tfrac N{\beta_n} \rfloor,1\ldots,1)$.

We consider first the case $q>p$. By \eqref{E:EK lower:q > p}
\begin{align*}
\|\scriptE_{\beta}^{\tilde{k}}\|_{L^p \to L^q}^{\rm{RWT}} &\gtrsim \bigg[\big(\prod_{i=1}^j 2^{-k_i\beta_i[\frac 2{\beta_i}]}\big)2^{k_{j+1}\beta_{j+1}[(1-\theta)-\frac 2{\beta_{j+1}}]}\big(\prod_{m=j+2}^d 2^{k_i\beta_i[1-\frac 2{\beta_{i}}]}\big)\bigg]^{\frac1{p'}-\frac1q}
\\
&\approx 2^{-N[\frac 2{\beta_1}+\cdots+\frac 2{\beta_{n}}-(n-j-\theta)](\frac 1{p'}-\frac 1q)}.
\end{align*}
Thus, by choosing $N$ large, we can make this term arbitrarily large if 
$$
2J_n-(n-j-\theta)<0,
$$
which, after a little algebra, is equivalent to
$$
\frac q{p'}<1+\frac{1}{J_{n}+\frac{d-n}{2}}.
$$

Next, we suppose $q\leq p$. By \eqref{E:EK lower:q leq p d>1}
\begin{equation} \label{E:q leq p cex}
\begin{aligned} 
\|\scriptE_{\beta}^{\tilde{k}}\|_{L^p \to L^q}^{\rm{RWT}} \gtrsim &\bigg(\prod_{i=1}^j 2^{-2k_{i}(\frac 1{p'}-\frac 1q)}\bigg)2^{-2k_{j+1}(\frac 1{p'}-\frac 1q)+k_{j+1}\beta_{j+1}(1+\theta)(1-\frac 2q)} 
\\
&\times\bigg(\prod_{i=j+2}^d 2^{k_{i}\beta_{i}[(1-\frac 2q)-\frac 2{\beta_{i}}(\frac 1{p'}-\frac 1q)]}\bigg)\tilde \alpha(\tfrac{k_{j+1}\beta_{j+1}}{k_{d}\beta_{d}}).
\end{aligned}
\end{equation}
Thus, for all $n$ such that $j+\theta\leq n<d$, 
$$
\|\scriptE_{\beta}^{\tilde{k}}\|_{L^p \to L^q}^{\rm{RWT}} \gtrsim 2^{-N[2(\frac 1{\beta_1}+\cdots+\frac 1{\beta_{n}})(\frac 1{p'}-\frac 1q) -(n-j-\theta)(1-\frac 2q)]}\tilde \alpha(N),
$$
which we can make arbitrarily large, for large $N$, if
$$
2J_n(\tfrac 1{p'}-\tfrac 1q) -(n-j-\theta)(1-\tfrac 2q)\leq 0,
$$
which, after a little algebra, is equivalent to
$$
\tfrac{1+J_n+d-n}q\geq \tfrac{J_n}{p'}+\tfrac{d-n}{2}.
$$
In the case where $n=d$, \eqref{E:q leq p cex} becomes
$$
\|\scriptE_{\beta}^{\tilde{k}}\|_{L^p \to L^q}^{\rm{RWT}} \gtrsim 2^{-N[2(\frac 1{\beta_1}+\cdots+\frac 1{\beta_{d}})(\frac 1{p'}-\frac 1q) -(d-j-\theta)(1-\frac 2q)]},
$$
which we can make arbitrarily large, for large $N$, if
$$
\tfrac{1+J_d}q> \tfrac{J_d}{p'}.
$$
\par Lastly, we consider the case where conditions (i-iii) fail, but condition (iv) holds, implying that $q<p$, $\tfrac{1+J_d}q= \tfrac{J_d}{p'}$, and $\beta_i> 2$ for all $i$. 

Let $\vec{k}_m=(\frac{Mm}{\beta_1},...,\frac{Mm}{\beta_d})$, where $M>100\max(\beta_1,...,\beta_d)$, and let $\varphi_{R^{\vec{k}_m}}$ be a Schwartz function supported on $R^{\vec{k}_m}$ and satisfying $0\leq\varphi_{R^{\vec{k}_m}}\leq 1$ and $\int\varphi_{R^{\vec{k}_m}}\approx |R^{\vec{k}_m}|\approx 2^{-2MmJ_d}$. Then $|\scriptE_\beta\varphi_{R^{\vec{k}_m}}|\gtrsim |R^{\vec{k}_m}|$ on some dual rectangle $R^*_{\vec{k}_m}$, of dimensions $2^{\frac{2Mm}{\beta_1}}\times\cdots\times 2^{\frac{2Mm}{\beta_1}}\times 2^{2Mm}$, and decays rapidly away from $R^*_{\vec{k}_m}$.

Define
$f(\xi)=\sum_{m=1}^N e^{i\vec{x}_m\cdot \xi}2^{-\frac{2MmJ_d}{p}}\varphi_{R^{\vec{k}_m}}$, with $\vec{x}_m$ chosen so that the dual rectangles $R^*_{\vec{k}_m}$ are widely separated. Then
$$
||f||_{L^p}\approx \bigg(\sum_{m=1}^N 2^{-2MmJ_d}|R^{\vec{k}_m}|\bigg)^\frac{1}{p}=N^\frac{1}{p},
$$
and
\begin{align*}
||\scriptE_\beta f||_{L^q}&\gtrsim \bigg(\sum_{m=1}^N 2^{\frac{2MmJ_dq}{p}}|R^{\vec{k}_m}|^q|R^*_{\vec{k}_m}| \bigg)^\frac{1}{q}=\bigg(\sum_{m=1}^N 2^{-2MmJ_dq(1-\frac{1}{p})}2^{2Mm(J_d+1)} \bigg)^\frac{1}{q}
\\
&=\bigg(\sum_{m=1}^N 2^{-\frac{2MmJ_dq}{p'}}2^{\frac{2MmJ_dq}{p'}} \bigg)^\frac{1}{q}=N^\frac 1q,
\end{align*}
where in the last line, we used $\tfrac{1+J_d}q= \tfrac{J_d}{p'}$. Therefore,
$||\scriptE_\beta||_{L^p\rightarrow L^q}\gtrsim N^{\frac 1q-\frac 1p}$, which goes to infinity as $N\rightarrow \infty$, since $q<p$. Thus, $\scriptE_\beta$ fails to have a strong type bound.

\end{proof}

\end{document}